\documentclass[12pt,a4paper]{amsart}
\usepackage{amsmath,amssymb,latexsym,tabularx}
\usepackage[dvips]{graphicx}
\usepackage[colorlinks=true,linkcolor=blue,citecolor=red]{hyperref}
\usepackage[usenames,dvipsnames]{color}
\headsep=1.2500cm \numberwithin{equation}{section}
\newtheorem{theorem}{Theorem}[section]
\newtheorem{proposition}[theorem]{Proposition}
\newtheorem{corollary}[theorem]{Corollary}
\newtheorem{lemma}[theorem]{Lemma}
\newtheorem{question}[theorem]{Question}
\theoremstyle{remark}

\newtheorem{definition}[theorem]{Definition}
\newtheorem{example}[theorem]{Example}

\newcommand{\asdim}{\operatorname{asdim}}

\title[A concept for resemblance in large scale geometry]{A concept for resemblance in large scale geometry}

\author[Sh. Kalantari]{Sh. Kalantari}
 \address[Sh. Kalantari]{Department of Basic Sciences, Babol Noshirvani University of Technology, Shariati Ave.,Babol, Iran, Post Code:47148-71167.}
 \email{shahab.kalantari@nit.ac.ir}

%Extension of maps; metric geometry (others); Categories of topological space and continuous mappings; Global geometric and topological methods (a la Gromov), differential geometric analysis on metric spaces

\begin{document}
\maketitle
\begin{abstract}
In this paper, we introduce the notion of \emph{large scale resemblance structure} as a new large scale structure by axiomatizing the concept of \emph{being alike in large scale} for a family of subsets of a set. We see that in a particular case, large scale resemblances on a set can induce a \emph{nearness} on it, and as a consequence, we offer a relatively big class of examples to show that \emph{not every near family is contained in a bunch}. Besides, We show how some large scale properties like \emph{asymptotic dimension} can be generalized to large scale resemblance spaces.
\end{abstract}
%\tableofcontents
\section{introduction}
In 1951 Efremovic defined the concept of \emph{proximity} as a \emph{small scale structure} on sets. He, in fact, tried to axiomatize the notion of \emph{being near} for two subsets $A$ and $B$ of a set $X$ (\cite{Ef1},\cite{Ef2}). As a generalization for the concept of proximity, Herrlich defined the notion of \emph{nearness} (\cite{her}). Nearness, as a small scale structure, tries to axiomatize the concept of being near for a family of subsets of a set $X$, and it somehow unifies other small scale structures like \emph{uniformity}, proximity and \emph{contiguity}:\\
`The category of all nearness spaces and nearness maps contains categories of these small scale structures and their preserving maps as embedded subcategories with some nice properties (\cite{her}).'\\
It is worth mentioning that each small scale structure has its advantages and difficulties, and our purposes and needs can lead us to use one of them. For more details about small scale structures and their relative applications, we recommend the reader to see \cite{hbb}, \cite{Nai2}, \cite{Nai}.\\
Recently investigating spaces in \emph{large scale} has drawn a significant amount of attention. Most large scale structures have been defined inspired by definitions of well-known small scale structures. For example, Roe defined the concept of \emph{coarse structures} inspired by the definition of uniformity (\cite{Roe}). Another example of a large scale structure is the concept of \emph{asymptotic resemblance relation} that has been defined influenced by the definition of proximity (\cite{me}). An asymptotic resemblance relation on a set $X$ tries to axiomatize the concept of \emph{being alike in large scale} for two subsets $A$ and $B$ of $X$. This paper is an attempt for axiomatizing the concept of being alike in large scale for a family of subsets of a set $X$. For this reason, we define the notion of \emph{large scale resemblance} and investigate some of its properties. Let us give an abbreviate overview of what happens in this paper.\\
In \S \ref{s1} we discuss briefly what is needed for understanding the rest of the paper. We give basic definitions and properties about near structures and proximities in small scale, and coarse structures and asymptotic resemblance relations in large scale. We introduce large scale resemblance spaces in \S \ref{s2}. Besides, basic definitions and properties about large scale resemblance spaces together with some examples of large scale resemblance spaces can be found in this section. In \S \ref{s3} we show how in some cases, a large scale resemblance structure on a set $X$ can induce a nearness structure on $X$. In addition, we show that the answer to the following question in `No' for a relatively big class of examples.\\
\begin{question}
Suppose that $(X,\mathfrak{N})$ is a near space and $\mathcal{A}\in \mathfrak{N}$. Is there any \emph{bunch} $\mathcal{C}$ in $(X,\mathfrak{N})$ such that $\mathcal{A}\subseteq \mathcal{C}?$
\end{question}
We show how the concept of \emph{asymptotic dimension} can be generalized to large scale resemblance spaces in \S \ref{s4}. In this section, we also show that large scale equivalent large scale resemblance spaces have the same asymptotic dimension. Finally, in \S \ref{s5} we define \emph{large scale regular} (\emph{LS-regular}) large scale resemblance spaces. We show in this section, the category of LS-regular large scale resemblance spaces and large scale mappings contains the category of all asymptotic resemblance spaces and AS.R mappings as an embedded reflective full subcategory.
\section{Preliminaries}\label{s1}
\subsection{Proximity and nearness}
Let us fix some notations first. For a nonempty set $X$, $\mathcal{P}(X)$ denotes the family of all subsets of $X$, so $\mathcal{P}(\mathcal{P}(X))$ denotes the family of all subsets of $\mathcal{P}(X)$. For two elements $\mathcal{A}$ and $\mathcal{B}$ of $\mathcal{P}(\mathcal{P}(X))$ we define
$$\mathcal{A}\vee \mathcal{B}=\{A\cup B\mid A\in \mathcal{A},B\in \mathcal{B}\}$$
In addition, assume that $\mathcal{A},\mathcal{B}\subseteq \mathcal{P}(X)$. We write $\mathcal{B}\ll \mathcal{A}$, if for each $A\in \mathcal{A}$ there exists some $B\in \mathcal{B}$ such that $B\subseteq A$.
\begin{definition}\label{near}
Let $X$ be a nonempty set. A \emph{near structure} or a \emph{nearness} on $X$ is a subset $\mathfrak{N}$ of $\mathcal{P}(\mathcal{P}(X))$ such that it has the following additional properties.\\
i) If $\bigcap_{A\in \mathcal{A}}A\neq \emptyset$ then $\mathcal{A}\in \mathfrak{N}$.\\
ii) If $\mathcal{A}\in \mathfrak{N}$ and $\mathcal{A}\ll \mathcal{B}$, then $\mathcal{B}\in \mathfrak{N}$.\\
iii) Each member of $\mathfrak{N}$ does not contain $\emptyset$.\\
iv) If $\mathcal{A},\mathcal{B}\notin \mathfrak{N}$ then $\mathcal{A}\vee \mathcal{B}\notin \mathfrak{N}$.\\
If $\mathfrak{N}$ is a nearness on $X$, then the pair $(X,\mathfrak{N})$ is called a \emph{N-space}.
\end{definition}
Let $(X,\mathfrak{N})$ be a N-space. For $A\subseteq X$ and $x\in X$ define $x\in \overline{A}$ if $\{x,A\}\in \mathfrak{N}$. We say $A\subseteq X$ is closed if $\overline{A}=A$. So we have a topology on $X$. This topology is $R_{0}$ and it is called the induced topology by $\mathfrak{N}$ on $X$.
To make this paper more self-content, let us recall some more definitions here.
\begin{definition}
Let $X$ be a set. Suppose that $\delta$ is a relation on $\mathcal{P}(X)$. For each $A,B\subseteq X$, denote $(A,B)\notin \delta$ by $A\bar{\delta}B$. Then the relation $\delta$ is called a \emph{proximity} on $X$ if it satisfies the following properties.\\
i) If $A\delta B$ then $B\delta A$.\\
ii) $A\delta (B\cup C)$ if and only if $A\delta B$ or $A\delta C$.\\
iii) If $A\delta B$ then $A\neq \emptyset$ and $B\neq \emptyset$.\\
iv) If $A\bar{\delta}B$ then there exists some $D\subseteq X$ such that $A\bar{\delta}D$ and $(X\setminus D)\bar{\delta}B$.\\
for all $A,B,C\subseteq X$. The pair $(X,\delta)$ is called a \emph{proximity space} (\cite{Ef1},\cite{Ef2}).
\end{definition}
Let $(X,\delta)$ be a proximity space and $A\subseteq X$. Define $x\in \overline{A}$ if $\{x\}\delta A$. If we assume $A\subseteq X$ is closed, if and only if, $\overline{A}=A$ then we have a topology on $X$ which is called the induced topology by $\delta$ on $X$.
\begin{definition}
Suppose that $(X,\delta)$ is a proximity space. A family $\mathcal{C}$ of subsets of $X$ is called a \emph{cluster} in $(X,\delta)$ if it satisfies the following properties.\\
i) $A\delta B$, for all $A,B\in \mathcal{C}$.\\
ii) $(A\cup B)\in \mathcal{C}$ if and only if $A\in \mathcal{C}$ or $B\in \mathcal{C}$.\\
iii) $A\delta B$ for all $B\in \mathcal{C}$ implies $A\in \mathcal{C}$.
\end{definition}
A proximity space $(X,\delta)$ is called \emph{separated} if $\{x\}\delta\{y\}$ implies $x=y$, for all $x,y\in X$. Suppose that $(X,\delta)$ is separated proximity space. Let $\mathfrak{X}$ denote the family of all clusters in $(X,\delta)$. For two subsets $\mathfrak{A}$ and $\mathfrak{B}$ of $\mathfrak{X}$, define $\mathfrak{A}\delta^{*}\mathfrak{B}$ if $A\subseteq X$ is in all elements of $\mathfrak{A}$ and $B\subseteq X$ is in all elements of $\mathfrak{B}$ then $A\delta B$. It can be shown that $(\mathfrak{X},\delta^{*})$ is a compact Hausdorff proximity space and it contains $X$ as an open dense subset. Thus $\mathfrak{X}$ is a compactification of $(X,\delta)$ and it is called the \emph{Smirnov compactification} of $X$. For more details about proximity spaces and their Smirnov compactifications see \cite{Nai2} and \cite{Nai}.
\subsection{Coarse structures and asymptotic resemblance relations}
\begin{definition}\label{defcoarse}
Let $X$ be a set and assume that $E,F\subseteq X\times X$. We define $$E^{-1}=\{(y,x)\in X\times X\mid (x,y)\in E\}$$ and $$E\circ F=\{(x,y)\in X\times X\mid \textrm{there exists some }z\in X\textrm{ such that }(x,z)\in F\,and\,(z,y)\in E\}$$
A subset $\mathcal{E}$ of $\mathcal{P}(X\times X)$ is called a \emph{coarse structure} on the set $X$, if it satisfies the following properties.\\
i) If $E\subseteq F$, for some $F\in \mathcal{E}$, then $E\in \mathcal{E}$.\\
ii) $E\circ F,E\cup F,E^{-1}\in \mathcal{E}$, for all $E,F\in \mathcal{E}$.\\
iii) $\Delta=\{(x,x)\mid x\in X\}\in \mathcal{E}$.\\
If $\mathcal{E}$ is a coarse structure on the set $X$ then the pair $(X,\mathcal{E})$ is called a \emph{coarse space}.
\end{definition}
\begin{example}
Suppose that $(X,d)$ is a metric space. Define $E\in \mathcal{E}_{d}$ if there exists some $r>0$ such that $d(x,y)\leq r$, for all $(x,y)\in E$. The family $\mathcal{E}_{d}$ is a coarse structure on $X$.
\end{example}
Suppose that $X$ is a set and $E\subseteq (X\times X)$. Let
$$E(A)=\{b\in X\mid (a,b)\in E\,for\,some\,a\in A\}$$
for all $A\subseteq X$.
\begin{example}\label{tt}
Let $\overline{X}$ be a compactification of the Hausdorff and locally compact topological space $X$. Define $\mathcal{E}\subseteq \mathcal{P}(X\times X)$ as follows.\\
`$E\in \mathcal{E}$ if $E(K)$ is relatively compact in $X$, for all relatively compact $K\subseteq X$, and if $(x_{\alpha},y_{\alpha})_{\alpha \in I}$ is a net in $E$ and $x_{\alpha}\rightarrow w\in (\overline{X}\setminus X)$ then $y_{\alpha}\rightarrow w$.'\\
Recall that a subset $K$ of $X$ is called relatively compact if $\overline{K}$ is compact. The family $\mathcal{E}$ defines a coarse structure on $X$ and it is called the \emph{topological coarse structure} associated to the compactification $\overline{X}$ of $X$ (see 2.2 of \cite{Roe}).
\end{example}
A coarse structure $\mathcal{E}$ on the topological space $X$ is called \emph{compatible} with the topology if there exists an open $E\in \mathcal{E}$ such that $\Delta \subseteq E$ and it is called \emph{proper} if each bounded subset of $(X,\mathcal{E})$ is relatively compact. The following question arises naturally from Example \ref{tt}.\\
`Let $X$ be a locally compact Hausdorff topological space with a compatible and proper coarse structure $\mathcal{E}$. Is there any compactification $\overline{X}$ of $X$ such that the topological coarse structure associated to $\overline{X}$ is equal to $\mathcal{E}$?'.\\
This question has a partial answer.
\begin{proposition}\label{hig}
Let $\mathcal{E}$ be a proper and compatible coarse structure on the Hausdorff topological space $X$. Then there exists a compactification $hX$ of $X$ such that the topological coarse structure associated to $hX$ contains $\mathcal{E}$. In addition, if $\overline{X}$ is another compactification of $X$ with this property, then the identity map extends uniquely to a continuous map from $hX$ to $\overline{X}$.
\end{proposition}
\begin{proof}
See Proposition 2.39 of \cite{Roe}.
\end{proof}
The compactification $hX$ in Proposition \ref{hig} is called the \emph{Higson compactification} of $X$. We denote the boundary $hX\setminus X$ by $\nu X$ and we call it the \emph{Higson corona} of $X$. For a way of constructing the Higson compactification of a topological space with a proper and compatible coarse structure, see \S 2.3 of \cite{Roe}.
\begin{definition}
Let $X$ be a set. An equivalence relation $\lambda$ on $\mathcal{P}(X)$ is called an \emph{asymptotic resemblance} (an AS.R) if,\\
i) $A_{i}\lambda B_{i}$ for $i\in \{1,2\}$ then $(A_{1}\cup B_{1})\lambda(A_{2}\cup B_{2})$.\\
ii) $(A_{1}\cup A_{2})\lambda B$ for nonempty subsets $A_{1},A_{2},B\subseteq X$ then there are $B_{1},B_{2}\neq \emptyset$ such that $B=B_{1}\cup B_{2}$ and $A_{i}\lambda B_{i}$, for $i\in \{1,2\}$.\\
In this case the pair $(X,\lambda)$ is called an \emph{asymptotic resemblance space} (AS.R space). For two subsets $A$ and $B$ of an AS.R space $(X,\lambda)$ if $A\lambda B$ then we say $A$ and $B$ are \emph{asymptotically alike}.
\end{definition}
\begin{example}
Suppose that $X$ is a set and let $\mathcal{E}$ denote a coarse structure on $X$. Define $A\lambda_{\mathcal{E}}B$ if there exists some $E\in \mathcal{E}$ such that $A\subseteq E(B)$ and $B\subseteq E(A)$. It can be easily seen that $\lambda_{\mathcal{E}}$ is an AS.R on $X$. If $\mathcal{E}=\mathcal{E}_{d}$ for some metric $d$ on $X$ then we denote $\lambda_{\mathcal{E}}$ by $\lambda_{d}$. In this case, if $A,B\subseteq X$ then $A\lambda_{d}B$ means that $A$ and $B$ has finite Hausdorff distance.
\end{example}
Let $(X,\lambda)$ be an AS.R space. A subset $D$ of $X$ is called \emph{bounded} if $D=\emptyset$ or $D\lambda\{x\}$, for some $x\in X$. Two subsets $A$ and $B$ of $X$ are called \emph{asymptotically disjoint}, if $L_{1}$ and $L_{2}$ are two unbounded subsets of $A$ and $B$, respectively, then they are not asymptotically alike. The AS.R space $(X,\lambda)$ is called to be \emph{asymptotically normal} if $A$ and $B$ are two asymptotically disjoint subsets of $X$, then there are subsets $X_{1}$ and $X_{2}$ of $X$ such that $X=X_{1}\cup X_{2}$ and they are asymptotically disjoint from $A$ and $B$, respectively. Let us recall that if $d$ is a metric on the set $X$, then $D$ is bounded in the AS.R space $(X,\lambda_{d})$, if and only if, $D$ is bounded in the metric space $(X,d)$. It can be shown that if $(X,d)$ is a metric space, then $(X,\lambda_{d})$ is an asymptotically normal AS.R space (Proposition 4.5 of \cite{me}).
\begin{definition}\label{pad}
Let $\mathcal{E}$ be a compatible and proper coarse structure on the normal topological space $X$. Assume that $(X,\lambda_{\mathcal{E}})$ is an asymptotically normal AS.R space. Define\\
`$A\delta B$ if $\overline{A}\cap \overline{B}\neq \emptyset$ or $A$ and $B$ are not asymptotically disjoint in $(X,\lambda_{\mathcal{E}})$.'
\end{definition}
It can be shown that the relation $\delta$ defined in Definition \ref{pad} is a separated proximity on the normal topological space $X$ which is compatible with the topology of $X$ and the Smirnov compactification of $(X,\delta)$ is homeomorphic to the Higson compactification of the coarse space $(X,\mathcal{E})$ (for details see \S 5 of \cite{me}).\\
We can end our preliminaries by the definition of \emph{asymptotic dimension}. The notion of asymptotic dimension of metric spaces has been introduced in \cite{gro}. For more detailed information about this concept, we recommend the reader to see \cite{asdim}. First recall that a family $\mathcal{U}$ of subsets of the coarse space $(X,\mathcal{E})$ is called \emph{uniformly bounded} if there exists some $E\in \mathcal{E}$ such that $U\times U\subseteq E$, for all $U\in \mathcal{U}$. If $\mathcal{E}=\mathcal{E}_{d}$, for some metric $d$ on $X$, we can equivalently say that $\mathcal{U}$ is uniformly bounded, if and only if, there exists some $r>0$ such that $\operatorname{diam}(U)\leq r$, for all $U\in \mathcal{U}$. For the case of AS.R spaces, we have the following definition.
\begin{definition}
Suppose that $(X,\lambda)$ is an AS.R space and $\mathcal{U}$ is a family of subsets of $X$. Let
$$\bigotimes_{\mathcal{U}}=\bigcup_{U\in \mathcal{U}}U\times U$$
We call the family $\mathcal{U}$ uniformly bounded if $A\subseteq \bigotimes_{\mathcal{U}}(B)$ and $B\subseteq \bigotimes_{\mathcal{U}}(A)$ implies $A\lambda B$, for all $A,B\subseteq X$.
\end{definition}
It can be shown that if $(X,d)$ is a metric space then a family $\mathcal{U}$ of subsets of $X$ is uniformly bounded in the AS.R space $(X,\lambda_{d})$, if and only if, $\mathcal{U}$ is uniformly bounded it the coarse space $(X,\mathcal{E}_{d})$ (see Proposition 6.2 of \cite{me}). Let $X$ be a set and assume that $\mathcal{U},\mathcal{V}\subseteq \mathcal{P}(X)$. Recall that `$\mathcal{U}$ refines $\mathcal{V}$' means that for each $U\in \mathcal{U}$ there exists some $V\in \mathcal{V}$ such that $U\subseteq V$. In addition, recall that the \emph{multiplicity} of a cover $\mathcal{U}$ of $X$ is the smallest natural number $n$ (if such a $n$ exists) such that each $x\in X$ belongs to at most $n$ elements of $\mathcal{U}$. Several equivalent definitions can be found in the literature for the asymptotic dimension of metric spaces and coarse spaces (see \cite{asdim}, \cite{Gra}). Since the following definition can be applied for AS.R spaces too, it is our favourite one.
\begin{definition}
Let $(X,d)$ be a metric space. We say that the \emph{asymptotic dimension} of $X$ is less than or equal to $n\in \mathbb{N}$ if for each uniformly bounded cover $\mathcal{U}$ of $X$ there exists a uniformly bounded cover $\mathcal{V}$ of $X$ such that $\mathcal{U}$ refines $\mathcal{V}$ and the multiplicity of $\mathcal{V}$ is less than or equal to $n+1$. If the asymptotic dimension of the metric space $(X,d)$ is less than or equal to $n$, but it is not less than $n-1$, then we say that the asymptotic dimension of $X$ is equal to $n$.
\end{definition}
We have definitions of asymptotic dimensions of coarse spaces and AS.R spaces if we substitute the word `metric space' with `coarse space' and `AS.R space', respectively.
\section{Large scale resemblance}\label{s2}
The following definition is the main definition of this paper.
\begin{definition}\label{defasl}
Let $X$ be a nonempty set. We call a subset $\mathfrak{C}$ of $\mathcal{P}(\mathcal{P}(X))$ a \emph{large scale resemblance} (LS.R) on $X$ if it has the following properties.\\
i) $\{A\}\in \mathfrak{C}$, for all $A\subseteq X$.\\
ii) If $\mathcal{B}\subseteq \mathcal{A}$, for some $\mathcal{A}\in \mathfrak{C}$, then $\mathcal{B}\in \mathfrak{C}$.\\
iii) If $\mathcal{A},\mathcal{B}\in \mathfrak{C}$ and $\mathcal{A}\cap \mathcal{B}\neq \emptyset$, then $\mathcal{A}\cup \mathcal{B}\in \mathfrak{C}$.\\
iv) If $\mathcal{A},\mathcal{B}\in \mathfrak{C}$, then $\mathcal{A}\vee \mathcal{B}\in \mathfrak{C}$.\\
If $\mathfrak{C}$ denotes a LS.R on the set $X$, then we call the pair $(X,\mathfrak{C})$ a \emph{large scale resemblance space} (LS.R space).
\end{definition}
\begin{example}\label{counter}
Let $X=\{a,b,c\}$. Assume that
$$\mathfrak{C}=\{\{A\}\mid A\subseteq X\}\cup \{\{\{a\},\{a,b\}\},\{\{a,c\},\{a,b,c\}\}\}$$
It is easy to see that $\mathfrak{C}$ is a LS.R on $X$.
\end{example}
\begin{example}
Let $d$ be a metric on the set $X$. Define the subset $\mathfrak{C}_{d}$ of $\mathcal{P}(\mathcal{P}(X))$ as follows.\\
`$\mathcal{A}\in \mathfrak{C}_{d}$ if there exists some $k>0$ such that $d_{H}(A,B)\leq k$, for each $A,B\in \mathcal{A}$.'\\
This subset $\mathfrak{C}_{d}$ of $\mathcal{P}(\mathcal{P}(X))$ is a LS.R on $X$ and we call it the LS.R induced by the metric $d$ on $X$. Instead of proving the claim of this example, we prove a much more general result in Example \ref{c1}.
\end{example}
\begin{example}\label{c1}
Let $\mathcal{E}$ be a coarse structure on the nonempty set $X$. We define the subset $\mathfrak{C}_{\mathcal{E}}$ of $\mathcal{P}(\mathcal{P}(X))$ as follows.\\
`$\mathcal{A}\in \mathfrak{C}_{\mathcal{E}}$ if there exists some $E\in \mathcal{E}$ such that for each $A,B\in \mathcal{A}$, $A\subseteq E(B)$.'\\
Properties i) and ii) of Definition \ref{defasl} hold evidently by using Definition \ref{defcoarse}. For the property iii) of Definition \ref{defasl} suppose that $\mathcal{A},\mathcal{B}\in \mathfrak{C}_{\mathcal{E}}$ and $C\in \mathcal{A}\bigcap \mathcal{B}$. Suppose that $E,F\in \mathcal{E}$ are such that $A\subseteq E(A^{\prime})$ and $B\subseteq F(B^{\prime})$, for all $A,A^{\prime}\in \mathcal{A}$ and $B,B^{\prime}\in \mathcal{B}$. Since $A\subseteq E(C)$ and $C\subseteq E(A)$, for all $A\in \mathcal{A}$, and $B\subseteq F(C)$ and $C\subseteq F(B)$, for all $B\in \mathcal{B}$, we clearly have $A\subseteq (E\circ F\cup F\circ E)(B)$ for all $A,B\in (\mathcal{A}\cup \mathcal{B})$. Thus $\mathcal{A}\cup \mathcal{B}\in \mathfrak{C}_{\mathcal{E}}$. Since the union of two elements of $\mathcal{E}$ is a member of $\mathcal{E}$, it is straightforward to verify the property iv) of Definition \ref{defasl}. Therefore $\mathfrak{C}_{\mathcal{E}}$ is a large scale resemblance on $X$ and we call it the LS.R induced by the coarse structure $\mathcal{E}$ on $X$.
\end{example}
\begin{example}
Let $(X,\mathcal{E})$ be a coarse space. Define $\tilde{\mathfrak{C}}_{\mathcal{E}}$ as follows.\\
`$\mathcal{A}\in \tilde{\mathfrak{C}}_{\mathcal{E}}$ if for each $A,B\in \mathcal{A}$ there exists some $E\in \mathcal{E}$ such that $A\subseteq E(B)$ and $B\subseteq E(A)$'.\\
Similar arguments to Example \ref{c1} can easily show that $\tilde{\mathfrak{C}}_{\mathcal{E}}$ is a LS.R on $X$.
\end{example}
\begin{example}
Let $(X,\lambda)$ be an AS.R space. We define $\mathfrak{C}_{\lambda}\subseteq \mathcal{P}(\mathcal{P}(X))$ as follows.\\
`$\mathcal{A}\in \mathfrak{C}_{\lambda}$ if $A\lambda B$, for all $A,B \in \mathcal{A}$.'\\
Properties i), ii) and iii) of Definition \ref{defasl} are straightforward consequences of the fact that $\lambda$ is an equivalence relation on $\mathcal{P}(X)$. Since for all $A,B,A^{\prime},B^{\prime}\subseteq X$, $A\lambda A^{\prime}$ and $B\lambda B^{\prime}$ implies $(A\cup A^{\prime})\lambda (B\cup B^{\prime})$, property iv) of Definition \ref{defasl} easily holds. Thus $\mathfrak{C}_{\lambda}$ is a LS.R on $X$. We call $\mathcal{C}_{\lambda}$ the LS.R induced by $\lambda$.
\end{example}
Recall that a coarse structure $\mathcal{E}$ on the set $X$ is called \emph{connected} if for each $x,y\in X$ there exists some $E\in \mathcal{E}$ such that $(x,y)\in E$.
\begin{corollary}
Let $\mathcal{E}$ be a coarse structure on the nonempty set $X$ and let $\lambda=\lambda_{\mathcal{E}}$. Then\\
i) $\mathfrak{C}_{\mathcal{E}}\subseteq \tilde{\mathfrak{C}}_{\mathcal{E}}$. If $\mathcal{E}$ is a connected coarse structure, then $\tilde{\mathfrak{C}}_{\mathcal{E}}= \mathfrak{C}_{\mathcal{E}}$, if and only if, $X$ is bounded. Besides, if $X$ is bounded then $\mathfrak{C}_{\mathcal{E}}=\mathcal{P}(\mathcal{P}(X))$.\\
ii) $\tilde{\mathfrak{C}}_{\mathcal{E}}=\mathfrak{C}_{\lambda}$.
\end{corollary}
\begin{proof}
It is straightforward.
\end{proof}
\begin{lemma}\label{unifb}
Suppose that $\mathcal{U}$ and $\mathcal{V}$ are two uniformly bounded families of subsets of the AS.R space $(X,\lambda)$. Then
$$\mathcal{W}=\{U\cup V\mid U\in \mathcal{U},V\in \mathcal{V} \textrm{ and } U\cap V\neq \emptyset\}$$
is a uniformly bounded family of subsets of $(X,\lambda)$.
\end{lemma}
\begin{proof}
Suppose that $A,B\subseteq X$ and $A\subseteq \bigotimes_{\mathcal{W}}(B)$ and $B\subseteq \bigotimes_{\mathcal{W}}(A)$. So $$A\subseteq (\bigotimes_{\mathcal{U}}(B)\cup \bigotimes_{\mathcal{V}}(B)\cup \bigotimes_{\mathcal{U}}(C_{1})\cup \bigotimes_{\mathcal{V}}(C_{2}))$$ where $C=\bigcup_{U\in \mathcal{U},V\in \mathcal{V}}(U\cap V)$ and $C_{1}=\bigotimes_{\mathcal{V}}(B)\cap C$ and $C_{2}=\bigotimes_{\mathcal{U}}(B)\cap C$. Since $\mathcal{U}$ and $\mathcal{V}$ are uniformly bounded it can be easily shown that there are subsets $B_{1},B_{2},B_{3},B_{4}$ of $B$ such that they are asymptotically alike to $\bigotimes_{\mathcal{U}}(B), \bigotimes_{\mathcal{V}}(B),\bigotimes_{\mathcal{U}}(C_{1}), \bigotimes_{\mathcal{V}}(C_{2})$, respectively. Thus there exists a subset $D$ of $X$ such that $A\subseteq D$ and $D\lambda B^{\prime}$ where $B^{\prime}=B_{1}\cup B_{2}\cup B_{3}\cup B_{4}$. Therefore $A$ is asymptotically alike to a subset of $B$ and one can similarly show that $B$ is asymptotically alike to a subset of $A$. The combination of these two facts can easily show that $A$ and $B$ are asymptotically alike.
\end{proof}
\begin{example}
Let $\lambda$ be an AS.R on the set $X$. Define $\tilde{\mathfrak{C}}_{\lambda}$ as follows.\\
`$\mathcal{A}\in \tilde{\mathfrak{C}}_{\lambda}$ if there exists a uniformly bounded family $\mathcal{U}$ of subsets of $X$ such that $A\subseteq \bigotimes_{\mathcal{U}}(B)$, for each $A,B\in \mathcal{A}$.'\\
To see the first property of Definition \ref{defasl} it suffices to notice that the family $\mathcal{U}=\{\{x\}\mid x\in X\}$ is uniformly bounded. Property ii) of Definition \ref{defasl} is an evident consequence of the definition of $\tilde{\mathfrak{C}}_{\lambda}$. Suppose that $\mathcal{A},\mathcal{B}\in \tilde{\mathfrak{C}}_{\lambda}$. Assume that $\mathcal{U}$ and $\mathcal{V}$ denote uniformly bounded families of subsets of $X$ such that they have the property mentioned above for $\mathcal{A}$ and $\mathcal{B}$ respectively. If $\mathcal{A}\cap \mathcal{B}\neq \emptyset$ one can easily see that the uniformly bounded family $\mathcal{W}$ mentioned in Lemma \ref{unifb} shows $(\mathcal{A}\cup \mathcal{B})\in \tilde{\mathfrak{C}}_{\lambda}$. Property iv) is a straightforward consequence of the fact $\mathcal{U}\cup \mathcal{V}$ is a uniformly bounded family of subsets of $X$.
\end{example}
\begin{corollary}
Let $\lambda$ be an AS.R on the set $X$. Then,\\
i) $\tilde{\mathfrak{C}}_{\lambda}\subseteq \mathfrak{C}_{\lambda}$.\\
ii) If $\lambda$ is the AS.R induced by the metric $d$ on $X$, then $\tilde{\mathfrak{C}}_{\lambda}=\mathfrak{C}_{d}$.\\
iii) If $\lambda$ is the AS.R induced by the coarse structure $\mathcal{E}$ on $X$, then $\mathfrak{C}_{\mathcal{E}}\subseteq \tilde{\mathfrak{C}}_{\lambda}\subseteq \tilde{\mathfrak{C}}_{\mathcal{E}}=\mathfrak{C}_{\lambda}$.
\end{corollary}
\begin{proof}
Parts i) and iii) are straightforward. Part ii) is an immediate consequence of Proposition 6.2 of \cite{me}.
\end{proof}
\begin{definition}
Let $(X,\mathfrak{C})$ be a LS.R space. We call a subset $B$ of $X$ \emph{bounded}, if $B=\emptyset$ or there exists some $x\in X$ such that $\{B,\{x\}\}\in \mathfrak{C}$. We call the LS.R space $(X,\mathfrak{C})$ \emph{connected}, if $\{\{x\},\{y\}\}\in \mathfrak{C}$, for each $x,y\in X$.
\end{definition}
\begin{lemma}\label{bound}
Let $(X,\mathfrak{C})$ be a connected LS.R space. Then the union of two bounded subsets of $X$ is a bounded subset of $X$.
\end{lemma}
\begin{proof}
Assume that $B_{1},B_{2}\subseteq X$ are bounded. Let $x,y\in X$ be such that\\ $\{B_{1},\{x\}\},\{B_{2},\{y\}\}\in \mathfrak{C}$. Since $(X,\mathfrak{C})$ is connected, $\{\{x\},\{y\}\}\in \mathfrak{C}$. Thus by property iii) of Definition \ref{defasl}, $$\mathcal{A}=\{B_{1},\{x\},\{y\}\},\mathcal{B}=\{B_{2},\{x\},\{y\}\}\in \mathfrak{C}$$ Property iv) of Definition \ref{defasl} shows that $\mathcal{A}\vee \mathcal{B}\in \mathfrak{C}$. Thus property ii) of Definition \ref{defasl} clearly implies that $B_{1}\cup B_{2}$ is bounded.
\end{proof}
\begin{lemma}\label{unb}
Let $(X,\mathfrak{C})$ be a LS.R space and assume that $L\subseteq X$ is unbounded. If $L^{\prime}\subseteq X$ and $L,L^{\prime}\in \mathcal{A}$, for some $\mathcal{A}\in \mathfrak{C}$, then $L^{\prime}$ is also unbounded.
\end{lemma}
\begin{proof}
Suppose that, contrary to our claim, $L^{\prime}$ is bounded. So there exists some $x\in X$ such that $\mathcal{B}=\{L^{\prime},\{x\}\}\in \mathfrak{C}$. Since $L^{\prime}\in \mathcal{A}\cap \mathcal{B}$, properties ii) and iii) of Definition \ref{defasl} show that $\{L,\{x\}\}\in \mathfrak{C}$, which contradicts our assumption that $L$ is unbounded.
\end{proof}
\begin{definition}
Let $(X,\mathfrak{C})$ and $(Y,\mathfrak{C}^{\prime})$ be two LS.R spaces. We call the map $f:X\rightarrow Y$ a \emph{large scale resemblance mapping} (LS.R mapping) if the inverse image of each bounded subset of $Y$ is a bounded subset of $X$ and $f(\mathcal{A})\in \mathfrak{C}^{\prime}$, for all $\mathcal{A}\in \mathfrak{C}$, where
 $$f(\mathcal{A})=\{f(A)\mid A\in \mathcal{A}\}$$
 We call a LS.R mapping $f:X\rightarrow Y$ a \emph{large scale equivalence} if there exists a LS.R mapping $g:Y\rightarrow X$ such that if $g\circ f(\mathcal{A}) \in \mathfrak{C}$ then $g\circ f(\mathcal{A})\cup \mathcal{A}\in \mathfrak{C}$ and if $f\circ g(\mathcal{B}) \in \mathfrak{C}^{\prime}$ then $f\circ g(\mathcal{B})\cup \mathcal{B}\in \mathfrak{C}^{\prime}$, for all $\mathcal{A}\subseteq \mathcal{P}(X)$ and $\mathcal{B}\subseteq \mathcal{P}(Y)$. In this case we call $g$ a \emph{large scale inverse} of $f$. Two LS.R spaces are called \emph{large scale equivalent} if there exists a large scale equivalence between them.
\end{definition}
\begin{lemma}\label{equi}
Let $(X,\mathfrak{C})$ and $(Y,\mathfrak{C}^{\prime})$ be two LS.R spaces. Assume that $f:X\rightarrow Y$ is a large scale equivalence between $X$ and $Y$ and $g:Y\rightarrow X$ is a large scale inverse of $f$.\\
i) If $f\circ g(\mathcal{B})\in \mathfrak{C}^{\prime}$ then $\mathcal{B}\in \mathfrak{C}^{\prime}$ and if $g\circ f(\mathcal{A})\in \mathfrak{C}$ then $\mathcal{A}\in \mathfrak{C}$, for all $\mathcal{A}\subseteq \mathcal{P}(X)$ and $\mathcal{B}\subseteq \mathcal{P}(Y)$.\\
ii) If $\mathcal{B}\in \mathfrak{C}^{\prime}$, then $f\circ g(\mathcal{B})\cup \mathcal{B}\in \mathfrak{C}^{\prime}$, and if $\mathcal{A}\in \mathfrak{C}$, then $g\circ f(\mathcal{A})\cup \mathcal{A}\in \mathfrak{C}$.
\end{lemma}
\begin{proof}
Both parts are easy to verify.
\end{proof}
\begin{proposition}
Let $(X,d_{1})$ and $(Y,d_{2})$ be two metric spaces. Suppose that $\mathfrak{C}_{1}$ and $\mathfrak{C}_{2}$ are LS.Rs associated to metrics $d_{1}$ and $d_{2}$ on $X$ and $Y$, respectively. Then the map $f:X\rightarrow Y$\\
i) is a LS.R mapping, if and only if, it is a coarse map.\\
ii) is a large scale equivalence, if and only if, it is a coarse equivalence.
\end{proposition}
\begin{proof}
The proof is straightforward by using Theorem 2.3 and Proposition 2.16 of \cite{me}.
\end{proof}
\begin{proposition}
Let $(X,\mathcal{E})$ and $(Y,\mathcal{E}^{\prime})$ be two coarse spaces and assume that\\ $f:X\rightarrow Y$ is a coarse map. Then
$f$ is a LS.R mapping between LS.R spaces $(X,\mathfrak{C}_{\mathcal{E}})$  and $(Y,\mathfrak{C}_{\mathcal{E}^{\prime}})$ (LS.R spaces $(X,\tilde{\mathfrak{C}}_{\mathcal{E}})$ and $(Y,\tilde{\mathfrak{C}}_{\mathcal{E}^{\prime}})$).
\end{proposition}
\begin{proof}
It is easy to verify.
\end{proof}
\begin{proposition}
Suppose that $(X,\lambda)$ and $(Y,\lambda^{\prime})$ are two AS.R spaces and $f:X\rightarrow Y$. Then,\\
i) $f$ is a LS.R mapping from $(X,\mathfrak{C}_{\lambda})$ to $(Y,\mathfrak{C}_{\lambda^{\prime}})$, if and only if, it is an AS.R mapping from $(X,\lambda)$ to $(Y,\lambda^{\prime})$.\\
ii) $f$ is a large scale equivalence between LS.R spaces $(X,\mathfrak{C}_{\lambda})$ and $(Y,\mathfrak{C}_{\lambda^{\prime}})$, if and only if, it is an asymptotic equivalence between AS.R spaces $(X,\lambda)$ and $(Y,\lambda^{\prime})$.
\end{proposition}
\begin{proof}
It is straightforward.
\end{proof}
From now on we assume all LS.R spaces are connected.
\section{Nearness structures induced from large scale resemblances}\label{s3}
\begin{example}
Assume that $X$ is a dense subspace of the topological space $Y$. Define the subset $\mathfrak{C}_Y$ of $\mathcal{P}(\mathcal{P}(X))$ as follows.\\
`$\mathcal{A}\in \mathfrak{C}_{Y}$ if $\overline{A}\cap (Y\setminus X)=\overline{B}\cap (Y\setminus X)$, for each $A,B\in \mathcal{A}$.'\\
It is straightforward to show that $\mathfrak{C}_{Y}$ is a LS.R on $X$. We call $\mathfrak{C}_{Y}$ the \emph{topological LS.R} on $X$ associated to $Y$.
\end{example}
\begin{proposition}\label{subset}
Let $\mathcal{E}$ be a proper and compatible coarse structure on the topological space $X$ and assume that $hX$ denotes the Higson compactification of $X$. Then $$\mathfrak{C}_{\mathcal{E}},\tilde{\mathfrak{C}}_{\mathcal{E}}\subseteq \mathfrak{C}_{hX}$$
\end{proposition}
\begin{proof}
Let $A,B\subseteq X$ and assume that $A\subseteq E(B)$ and $B\subseteq E(A)$, for some $E\in \mathcal{E}$. Let $w\in (\overline{A}\cap \nu X)$. Then there exists a net $(x_{\alpha})_{\alpha \in I}$ in $A$ such that $x_{\alpha}\rightarrow w$. Since $A\subseteq E(B)$, there exists some $y_{\alpha}\in B$ such that $(y_{\alpha},x_{\alpha})\in E$, for each $\alpha \in I$. By Proposition \ref{hig}, $E$ is a member of the topological coarse structure associated to $hX$. Therefore $y_{\alpha}\rightarrow w$ by definition of the topological coarse structure. It shows that $w\in (\overline{B}\cap \nu X)$. Thus $\overline{A}\cap \nu X\subseteq \overline{B}\cap \nu X$. Similarly one can show that $\overline{B}\cap \nu X\subseteq \overline{A}\cap \nu X$. Therefore $\overline{A}\cap \nu X= \overline{B}\cap \nu X$. By using this fact, it is straightforward to show the claim of this proposition.
\end{proof}
\begin{proposition}
Let $d$ be a metric on the set $X$ and let $\lambda=\lambda_{d}$. Then $\mathfrak{C}_{\lambda}=\mathfrak{C}_{hX}$.
\end{proposition}
\begin{proof}
It is a direct consequence of Proposition 4.22 and Corollary 4.24 of \cite{me}.
\end{proof}
\begin{example}\label{count1}
Let $X=\mathbb{R}$ and assume $X$ with the discrete topology. Suppose that $Y$ is the one point compactification of $X$. Clearly $\mathcal{A}\in \mathfrak{C}_{Y}$, if and only if, all members of $\mathcal{A}$ are finite or all members of $\mathcal{A}$ are infinite. We claim that $\mathfrak{C}_{Y}\neq \mathfrak{C}_{\mathcal{E}}$, for all coarse structure $\mathcal{E}$ on $X$. Suppose that, contrary to our claim, $\mathcal{E}$ is a coarse structure on $X$ and $\mathfrak{C}_{Y}= \mathfrak{C}_{\mathcal{E}}$. Clearly $\{\{x\},E(x)\}\in \mathfrak{C}_{Y}=\mathfrak{C}_{\mathcal{E}}$, for all $E\in \mathcal{E}$ and $x\in X$. Thus $E(x)$ is finite for all $E\in \mathcal{E}$ and $x\in X$. Since $\{\mathbb{N},\mathbb{R}\}\in \mathfrak{C}_{Y}= \mathfrak{C}_{\mathcal{E}}$, there exists some $E\in \mathcal{E}$ such that $\mathbb{R}\subseteq E(\mathbb{N})$ and $\mathbb{N}\subseteq E(\mathbb{R})$. So $\mathbb{R}\subseteq \bigcup_{n\in \mathbb{N}}E(n)$. Since $E(n)$ is finite for all $n\in \mathbb{N}$, $\bigcup_{n\in \mathbb{N}}E(n)$ is countable and this contradicts the fact that $\mathbb{R}$ is uncountable. Therefore $\mathfrak{C}_{Y}\neq \mathfrak{C}_{\mathcal{E}}$ as we claimed. Completely similar argument can show that $\mathfrak{C}_{Y}\neq \tilde{\mathfrak{C}}_{\mathcal{E}}$, for all coarse structure $\mathcal{E}$ on $X$.
\end{example}
\begin{example}
Suppose that $X=\mathbb{R}$ and consider $X$ with the discrete topology as Example \ref{count1}. Define $\mathcal{E}\subseteq X\times X$ as follows.\\
`$E\in \mathcal{E}$ if $E(x)$ is finite, for each $x\in X$.'\\
It is straightforward to show that $\mathcal{E}$ is a proper and compatible coarse structure on the topological space $X$. Since two subsets $A$ and $B$ of $X$ are asymptotically disjoint, if and only if, one of them is finite, it is easy to see that $(X,\lambda_{\mathcal{E}})$ is asymptotically normal. Assume that $\mathcal{C}$ and $\mathcal{C}^{\prime}$ are two clusters in $hX\setminus X$ and $\mathcal{C}\neq \mathcal{C}^{\prime}$ (Definition \ref{pad}). Then there are infinite subsets $A\in \mathcal{C}$ and $B\in \mathcal{C}^{\prime}$ such that they are asymptotically disjoint, a contradiction. Therefore the Higson compactification of $X$ is the one point compactification. By Example \ref{count1} and Proposition \ref{subset}, $\mathfrak{C}_{\mathcal{E}}$ and $\tilde{\mathfrak{C}}_{\mathcal{E}}$ are proper subsets of $\mathfrak{C}_{hX}$.
\end{example}
\begin{definition}\label{nl}
Suppose that $\mathfrak{C}$ is a LS.R on the set $X$. Define the subset $\mathfrak{N}_{\mathfrak{C}}$ of $\mathcal{P}(\mathcal{P}(X))$ as follows.\\
`$\mathcal{A}\in \mathfrak{N}_{\mathfrak{C}}$ if $\bigcap_{A\in \mathcal{A}}\overline{A}\neq \emptyset$, or there exists some $\mathcal{B}\in \mathfrak{C}$ such that $\mathcal{B}$ does not contain any bounded subsets of $X$ and $\mathcal{B}\ll \mathcal{A}$.'
\end{definition}
Before going further let us set up a notation. Let $\mathcal{E}$ be a coarse structure on the set $X$ and $E\in \mathcal{E}$. For a subset $L$ of $X$, we denote the family of all $L^{\prime}\subseteq X$ such that $L\subseteq E(L^{\prime})$ and $L^{\prime}\subseteq E(L)$, by $\mathcal{N}_{E}(L)$.
\begin{corollary}
Let $(X,\mathcal{E})$ be a coarse space. Then $\mathcal{N}_{E}(L)\in \mathfrak{C}_{\mathcal{E}}$, for each $E\in \mathcal{E}$ and $L\subseteq X$. If $\mathcal{A}\in \mathfrak{C}_{\mathcal{E}}$ then there exists some $E\in \mathcal{E}$ such that $\mathcal{A}\subseteq \mathcal{N}_{E}(L)$, for all $L\in \mathcal{A}$.
\end{corollary}
\begin{proof}
Notice that if $A,B\in \mathcal{N}_{E}(L)$, then $A\subseteq E\circ E(B)$ and $B\subseteq E\circ E(A)$. By using this fact the proof of this corollary is straightforward.
\end{proof}
\begin{theorem}
Let $\mathcal{E}$ be a coarse structure on the topological space $X$ and let $\mathfrak{C}=\mathfrak{C}_{\mathcal{E}}$. Then $\mathfrak{N}_{\mathfrak{C}}$ is a nearness structure on $X$ and it induces the same topology to the topology of $X$.
\end{theorem}
\begin{proof}
 Properties i), ii) and iii) of Definition \ref{near} are easy to prove. We are going to prove the property iv) of Definition \ref{near}. Suppose that $\mathcal{A},\mathcal{B}\notin \mathfrak{N}_{\mathfrak{C}}$. Thus $\bigcap_{A\in \mathcal{A}}\overline{A}=\bigcap_{B\in \mathcal{B}}\overline{B}=\emptyset$ and this easily shows that $\bigcap_{C\in \mathcal{A}\vee \mathcal{B}}\overline{C}=\emptyset$. Suppose that there exists some $\mathcal{C}\in \mathfrak{C}$ such that $\mathcal{C}\ll \mathcal{A}\vee \mathcal{B}$. Let $L$ be an arbitrary element of $\mathcal{C}$. Assume that $E\in \mathcal{E}$ is such that $E=E^{-1}$ and $\mathcal{C}\subseteq \mathcal{N}_{E}(L)$. Let $A\in \mathcal{A}$ and $B\in \mathcal{B}$. Since there exists some $K\subseteq (A\cup B)$ such that $L\subseteq E(K)$, if $a\in L\setminus E(A)$ then we can choose some $b_{a}\in B$ such that $(a,b_{a})\in E$. Let $L_{B}^{A}=\{b_{a}\mid a\in L\setminus E(A)\}$. Let $L_{B}=\bigcup_{A\in \mathcal{A}} L_{B}^{A}$, for each $B\in \mathcal{B}$. We have $L_{B}^{A}\subseteq E(L\setminus E(A))$ and $(L\setminus E(A))\subseteq E(L_{C}^{A})$, for all $A\in \mathcal{A}$ and $B,C\in \mathcal{B}$. Thus $L_{B}^{A}\subseteq E\circ E(L_{C}^{A})$, for all $A\in \mathcal{A}$ and $B,C\in \mathcal{B}$. It shows that $L_{B}\subseteq E\circ E(L_{C})$, for all $B,C\in \mathcal{B}$. Therefore $\tilde{\mathcal{C}}=\{L_{B}\mid B\in \mathcal{B}\}\in \mathfrak{C}$. Since we assumed that $\mathcal{B}\notin \mathfrak{N}_{\mathfrak{C}}$, Lemma \ref{unb} shows that $\tilde{\mathcal{C}}$ does not contain any unbounded subset of $(X,\mathfrak{C})$. Thus $L_{B}$ is bounded for each $B\in \mathcal{B}$. Now, fix $B\in \mathcal{B}$. Suppose that $A\in \mathcal{A}$. Since $A\cup B$ has a subset in $\mathcal{N}_{E}(L)$, if $a\in (L\setminus E(L_{B}))$ then we can choose some $c_{a}\in A$ such that $(a,c_{a})\in E$. Let $R_{A}=\{c_{a}\mid a\in (L\setminus E(L_{B}))\}$. Clearly $R_{A}\subseteq E(L\setminus E(L_{B}))$ and $(L\setminus E(L_{B}))\subseteq E(R_{C})$ and thus $R_{A}\subseteq E\circ E(R_{C})$, for all $A,C\in \mathcal{A}$. It implies that $\mathcal{D}=\{R_{A}\mid A\in \mathcal{A}\}\in \mathfrak{C}$. Since $\mathcal{D}\ll \mathcal{A}$ and $\mathcal{A}\notin \mathfrak{N}_{\mathfrak{C}}$, $R_{A}$ is bounded for all $A\in \mathcal{A}$. Thus $L\setminus E(B)$ is bounded and hence $L=E(L_{B})\cup (L\setminus E(L_{B}))$ is bounded. We showed that if $\mathcal{C}\in \mathfrak{C}$ and $\mathcal{C}\ll \mathcal{A}\vee \mathcal{B}$, then each member of $\mathcal{C}$ is bounded. Therefore $\mathcal{A}\vee \mathcal{B}\notin \mathfrak{N}_{\mathfrak{C}}$. The fact that $\mathfrak{N}_{\mathfrak{C}}$ induces the same topology to the topology of $X$ is evident.
\end{proof}
Let $\mathcal{A}$ be a family of subsets of the topological space $X$. Define
$$\overline{\mathcal{A}}=\{\overline{A}\mid A\in \mathcal{A}\}$$
A nearness $\mathfrak{N}$ on the topological space $X$ is said to be a \emph{Herrlich nearness} (\emph{H-nearness}) if $\overline{\mathcal{A}}\in \mathfrak{N}$ then $\mathcal{A}\in \mathfrak{N}$.
\begin{corollary}
Let $\mathcal{E}$ be a compatible and proper coarse structure on the topological space $X$. If $\mathfrak{C}=\mathfrak{C}_{\mathcal{E}}$, then $\mathfrak{N}_{\mathfrak{C}}$ is a H-nearness.
\end{corollary}
\begin{proof}
Assume that $\overline{\mathcal{A}}\in\mathfrak{N}_{\mathfrak{C}}$. If $\bigcap_{A\in \mathcal{A}}\overline{A}\neq \emptyset$ then $\mathcal{A}\in \mathfrak{N}_{\mathfrak{C}}$, by Definition \ref{nl}. Now assume that there exists some $\mathcal{B}\in \mathfrak{C}$ such that each element of $\mathcal{B}$ is unbounded and $\mathcal{B}\ll \overline{\mathcal{A}}$. Choose an open $E\in \mathcal{E}$ such that it contains the diagonal and $E=E^{-1}$. It can be easily seen that $\overline{A}\subseteq E(A)$, for all $A\subseteq X$. Choose $F\in \mathcal{E}$ such that $B_{1}\subseteq F(B_{2})$, for all $B_{1},B_{2}\in \mathcal{B}$. For each $A\in \mathcal{A}$, choose $B\in  \mathcal{B}$ such that $B\subseteq A$ and let $L_{A}=E(B)\cap A$. Thus $L_{A}$ contains $B$ and hence it is unbounded. Since $B\subseteq E(A)$, clearly $B\subseteq E(L_{A})$ and $L_{A}\subseteq E(B)$. Let $A_{1},A_{2}\in \mathcal{A}$. It is easy to see that $L_{A_{1}}\subseteq E(F(E(L_{A_{2}})))$. Thus $E\circ F\circ E$ is the desired element in $\mathcal{E}$ that can show that $\mathcal{C}=\{L_{A}\mid A\in \mathcal{A}\}\in \mathfrak{C}$. Clearly $\mathcal{C}\ll \mathcal{A}$ and hence $\mathcal{A}\in \mathfrak{N}_{\mathfrak{C}}$.
\end{proof}
\begin{example}
Let $X$ and $Y$ be as Example \ref{count1} and assume that $\mathfrak{C}=\mathfrak{C}_{Y}$. It can be easily seen that $\mathcal{A}\in \mathfrak{N}_{\mathcal{C}}$, if and only if, $\bigcap_{A\in \mathcal{A}}A\neq \emptyset$ or all elements of $\mathcal{A}$ are infinite, for all $\mathcal{A}\subseteq \mathcal{P}(X)$. Suppose that $\mathcal{A},\mathcal{B}\notin \mathfrak{N}_{\mathfrak{C}}$. So $\bigcap_{A\in \mathcal{A}}A=\emptyset$ and $\bigcap_{B\in \mathcal{B}}B=\emptyset$ and it clearly shows that $\bigcap_{C\in \mathcal{A}\vee \mathcal{B}}C=\emptyset$. In addition, there are $A\in \mathcal{A}$ and $B\in \mathcal{B}$ such that $A$ and $B$ are finite subsets of $X$. Thus $(A\cup B)\in (\mathcal{A}\vee \mathcal{B})$ is finite and hence $\mathcal{A}\vee \mathcal{B}\notin \mathfrak{N}_{\mathfrak{C}}$. Even though $\mathfrak{C}\neq \mathfrak{C}_{\mathcal{E}}$, for each coarse structure $\mathcal{E}$ on $X$, we showed that $\mathfrak{N}_{\mathfrak{C}}$ is a nearness on $X$ (in fact $\mathfrak{N}_{\mathcal{C}}$ is a H-nearness).
\end{example}
\begin{proposition}\label{11}
Let $\mathcal{E}$ be a compatible and proper coarse structure on the topological space $X$ and denote $\mathfrak{C}_{\mathcal{E}}$ by $\mathfrak{C}$. Then $\mathcal{A}\in \mathfrak{N}_{\mathfrak{C}}$ implies $\bigcap_{A\in \mathcal{A}}\nu A \neq \emptyset$, where $\nu A=\overline{A}\cap \nu X$ and $\overline{A}$ is the closure of $A$ in $hX$.
\end{proposition}
\begin{proof}
It is a straightforward consequence of Proposition \ref{subset}.
\end{proof}
The inverse of Proposition \ref{11} is not true in general.
\begin{example}
Let $X=\mathbb{N}$ and consider $X$ with the standard metric induced from $\mathbb{R}$. Let $\mathfrak{C}=\mathfrak{C}_{d}$ and assume that $\delta$ is the proximity defined in Definition \ref{pad}, where $\mathcal{E}=\mathcal{E}_{d}$. For each $n\in \mathbb{N}$ suppose that
$$A_{n}=\{2^{nk}\mid k\in \mathbb{N}\}$$
Let $\mathcal{F}$ denote the family of all $B\subseteq X$ such that $A_{n}\subseteq B$, for some $n\in \mathbb{N}$. Clearly $\mathcal{F}$ is a filter in $X$. Let $\tilde{\mathcal{F}}$ be an ultrafilter containing $\mathcal{F}$ and let
$$\mathcal{C}=\{A\subseteq X\mid A\delta B\,for\,all\,B\in \tilde{\mathcal{F}}\}$$
The family $\mathcal{C}$ is a cluster in $(X,\delta)$ (see Theorem 5.8 of \cite{Nai}) and clearly $\mathcal{C}\in \nu X$. In addition $\mathcal{C}\in \nu A_{n}$ for all $n\in \mathbb{N}$. Therefore $\bigcap_{n\in \mathbb{N}}\nu A_{n}\neq \emptyset$. Now suppose that $L\subseteq A_{1}$ and $m=\min \{a\mid a\in L\}$. So $m=2^{k}$, for some $k\in \mathbb{N}$. Since if $n\rightarrow +\infty$ then $\mid 2^{n}-2^{k}\mid \rightarrow +\infty$, $d(m,A_{n})\rightarrow +\infty$, where $n\rightarrow +\infty$. This clearly shows that $\{A_{n}\mid n\in \mathbb{N}\}\notin \mathfrak{N}_\mathfrak{C}$.
\end{example}
\begin{definition}
Let $(X,\mathfrak{N})$ be a N-space. The nonempty subset $\mathcal{C}$ of $\mathcal{P}(X)$ is called a \emph{bunch} in $(X,\mathfrak{N})$ if it satisfies the following properties.\\
i) $\mathcal{C}\in \mathfrak{N}$.\\
ii) $A\cup B\in \mathcal{C}$, if and only if, $A\in \mathcal{C}$ or $B\in \mathcal{C}$, for all $A,B\subseteq X$.\\
iii) if $\overline{A}\in \mathcal{C}$ then $A\in \mathcal{C}$.\\
It is worth mentioning that a bunch $\mathcal{C}$ in $(X,\mathfrak{N})$ is called a \emph{near cluster} if $(\{\{A\}\}\cup \mathcal{C})\in \mathfrak{N}$ then $A\in \mathcal{C}$.
\end{definition}
It is known that if $A\delta B$ for two subsets $A$ and $B$ of the proximity space $(X,\delta)$ then there exists a cluster $\mathcal{C}$ in $(X,\delta)$ such that $A,B\in \mathcal{C}$ (\cite{Tho}). But this result can not be generalized to all nearness spaces. A counter example for this later fact is given in \cite{Nai3}. We are going to offer a relatively big class of counter examples here.
\begin{proposition}\label{2222}
Let $d$ be a metric on the set $X$ and let $\mathfrak{C}=\mathfrak{C}_{d}$. Assume that $\mathcal{A}\in \mathfrak{N}_{\mathfrak{C}}$ is such that $\bigcap_{A\in \mathcal{A}}\overline{A}=\emptyset$. Then each bunch in $(X,\mathfrak{N}_{\mathfrak{C}})$ does not contain $\mathcal{A}$.
\end{proposition}
\begin{proof}
Assume that, contrary to our claim, there exists some bunch $\mathcal{C}$ in $(X,\mathfrak{N}_{d})$ such that $\mathcal{A}\subseteq \mathcal{C}$. Since $\mathcal{C}\in \mathfrak{N}_{\mathfrak{C}}$, there exists some $\mathcal{B}\in \mathfrak{C}_{d}$ such that $\mathcal{B}$ does not contain any bounded subset of $X$ and $\mathcal{B}\ll \mathcal{C}$. Let $L\in \mathcal{B}$. By Lemma 4.2 of \cite{me3}, there are unbounded and asymptotically disjoint subsets $L_{1}$ and $L_{2}$ of $L$ in the AS.R space $(X,\lambda_{d})$. Since $(X,\lambda_{d})$ is an asymptotically normal AS.R space, there exist $X_{1},X_{2}\subseteq X$ such that $X=X_{1}\cup X_{2}$ and they are asymptotically disjoint from $L_{1}$ and $L_{2}$ in $(X,\lambda_{d})$, respectively. Thus $X_{1}$ and $X_{2}$ do not contain any subset with finite Hausdorff distance from $L$. Thus $X_{1},X_{2}\notin \mathcal{C}$ and hence $X\notin \mathcal{C}$, a contradiction.
\end{proof}
Suppose that $X$ is a dense subspace of the topological space $Y$. Define,\\
`$\mathcal{A}\in \mathfrak{N}_{Y}$ if $\bigcap_{A\in \mathcal{A}}\overline{A}\neq \emptyset$, for all $\mathcal{A}\subseteq \mathcal{P}(X)$'.\\
The nearness structure $\mathfrak{N}_{Y}$ is called the \emph{topological nearness} on $X$ associated to $Y$.
\begin{corollary}
Let $(X,d)$ be a metric space and assume that $\mathfrak{C}=\mathfrak{C}_{d}$. Then the nearness $\mathfrak{N}_{\mathfrak{C}}$ is not the topological nearness associated to $Y$, for each compactification $Y$ of $X$.
\end{corollary}
\begin{proof}
It is a straightforward consequence of Proposition \ref{2222}.
\end{proof}
\section{Asymptotic dimension of LS.R spaces}\label{s4}
\begin{definition}\label{unif}
Let $\mathfrak{C}$ be a LS.R on the set $X$. We say the family $\mathcal{U}$ of subsets of $X$ is \emph{uniformly bounded} in $(X,\mathfrak{C})$, if
$$\mathcal{A}_{\mathcal{V}}=\{A\subseteq X\mid A\subseteq \bigcup_{U\in \mathcal{V}}U\,and\, A\cap U\neq \emptyset \,for\,all\,U\in \mathcal{V}\}\in \mathfrak{C}$$
for all nonempty $\mathcal{V}\subseteq \mathcal{U}$.
\end{definition}
\begin{corollary}\label{ub}
Suppose that $\mathcal{U}$ is a uniformly bounded family of subsets of the LS.R space $(X,\mathfrak{C})$, then each element of $\mathcal{U}$ is bounded.
\end{corollary}
\begin{proof}
Assume that $U\in \mathcal{U}$ and $x\in U$. Let $\mathcal{V}=\{U\}$. Clearly $x,U\in \mathcal{A}_{\mathcal{V}}$ and property ii) of Definition \ref{defasl} clearly shows that $\{\{x\},U\}\in \mathfrak{C}$.
\end{proof}
Let us recall the following lemma from \cite{me}.
\begin{lemma}\label{haus}
Suppose that $(X,d)$ is a metric space. Let $(a_{n})_{n\in \mathbb{N}}$ and $(b_{n})_{n\in \mathbb{N}}$ are two sequences in $X$. If for each $I\subseteq \mathbb{N}$ the Hausdorff distance between $\{a_{i}\mid i\in I\}$ and $\{b_{i}\mid i\in I\}$ is finite, then there exists some $R>0$ such that $d(a_{n},b_{n})\leq R$, for all $n\in \mathbb{N}$.
\end{lemma}
\begin{proof}
It is Lemma 2.2 of \cite{me}.
\end{proof}
\begin{proposition}
Let $d$ be a metric on the set $X$. Then the family $\mathcal{U}$ of subsets of $(X,\mathfrak{C}_{d})$ is uniformly bounded, if and only if, there exists some $R>0$ such that $\operatorname{diam}(U)\leq R$, for all $U\in \mathcal{U}$.
\end{proposition}
\begin{proof}
Suppose that there exists some $R>0$ such that $\operatorname{diam}(U)\leq R$, for all $U\in \mathcal{U}$. Let $\mathcal{V}$ be a nonempty subset of $\mathcal{U}$. Suppose that $A,B\in \mathcal{A}_{\mathcal{V}}$. If $a\in A$ there exists some $U\in \mathcal{V}$ such that $a\in U$. Since $B\cap U\neq \emptyset$, there exists some $b\in (B\cap U)$. So $d(a,b)\leq R$. Similar argument holds for each element of $B$ and this shows that the Hausdorff distance between $A$ and $B$ is less than or equal to $R$. Thus $\mathcal{A}_{\mathcal{V}}\in \mathfrak{C}_{d}$. To prove the converse, suppose that contrary to the claim of this proposition, for each $n\in \mathbb{N}$ there exists some $U_{n}\in \mathcal{U}$ such that $\operatorname{diam}(U_{n})>n$. Thus there are $a_{n},b_{n}\in U_{n}$ such that $d(a_{n},b_{n})>n$, for each $n\in \mathbb{N}$. Assume that $I\subseteq \mathbb{N}$. Let $\mathcal{V}=\{U_{i}\mid i\in I\}$. Suppose that $A=\{a_{i}\mid i\in I\}$ and $B=\{b_{i}\mid i\in I\}$. Clearly $A,B\in \mathcal{A}_{\mathcal{V}}$. Since $\mathcal{A}_{\mathcal{V}}\in \mathfrak{C}_{d}$, there exists some $K>0$ such that $d_{H}(A,B)\leq K$. Now clearly Lemma \ref{haus} shows that there exists some $R>0$ such that $d(a_{n},b_{n})\leq R$, for all $n\in \mathbb{N}$, a contradiction.
\end{proof}
Let $X$ be a set and assume $\mathcal{U},\mathcal{V}\subseteq \mathcal{P}(X)$. Recall that `$\mathcal{U}$ refines $\mathcal{V}$' means that for each $U\in \mathcal{U}$ there exists some $V\in \mathcal{V}$ such that $U\subseteq V$. In addition, recall that the \emph{multiplicity} of a cover $\mathcal{U}$ of $X$ is the smallest natural number $n$ (if such a $n$ exists) such that each $x\in X$ belongs to at most $n$ elements of $\mathcal{U}$.
\begin{definition}
We say that the \emph{asymptotic dimension} of the LS.R space $(X,\mathfrak{C})$ is less than or equal to $n$ ($n\in \mathbb{N}\cup \{0\}$) if each uniformly bounded cover $\mathcal{U}$ of $X$ refines a uniformly bounded cover $\mathcal{V}$ of $X$ such that the multiplicity of $\mathcal{V}$ is less than or equal to $n+1$. In this case we write $\asdim_{\mathfrak{C}}X\leq n$. If $\asdim_{\mathfrak{C}}X\leq n$ and $\operatorname{asdim}_{\mathfrak{C}}X\leq n-1$ is not true, we say $\operatorname{asdim}_{\mathfrak{C}}X=n$. If $\operatorname{asdim}_{\mathfrak{C}}X\leq n$ does not hold for each $n\in \mathbb{N}$, we say that $(X,\mathfrak{C})$ has infinite asymptotic dimension.
\end{definition}
\begin{definition}
Let $(X,\mathfrak{C})$ be a LS.R space and let $Y$ be a nonempty subset of $X$. We define\\
`$\mathcal{A}\in \mathfrak{C}\mid_{Y}$ if $\mathcal{A}\in \mathfrak{C}$, for all $\mathcal{A}\subseteq \mathcal{P}(Y)$.'\\
The family $\mathfrak{C}\mid_{Y}$ is a LS.R on $Y$ and we call it the \emph{subspace} LS.R induced on $Y$.
\end{definition}
\begin{proposition}\label{subs}
Let $Y$ be a nonempty subset of the LS.R space $(X,\mathfrak{C})$. Then
$$\asdim_{\mathfrak{C}\mid_{Y}}Y\leq \operatorname{asdim}_{\mathfrak{C}}X$$
\end{proposition}
\begin{proof}
Suppose that $\asdim_{\mathfrak{C}}X\leq n$, for some $n\in \mathbb{N}\cup \{0\}$. Assume that $\mathcal{U}$ is a uniformly bounded cover of $Y$. Let $\mathcal{V}=\mathcal{U}\cup \{\{x\}\mid x\in X\setminus Y\}$. Suppose that $\mathcal{W}\subseteq \mathcal{V}$. Let
$$B=\{x\in X\setminus Y\mid \{x\}\in \mathcal{W}\}$$
Assume that $\mathcal{B}=\{A\cap Y\mid A\in \mathcal{A}_{\mathcal{W}}\}$ and let $\mathcal{W}^{\prime}=\{U\in \mathcal{U}\mid U\in \mathcal{W}\}$. It is easy to see that $\mathcal{B}= \mathcal{A}_{\mathcal{W}^{\prime}}$. Since $\mathcal{U}$ is a uniformly bounded cover of $Y$ and $\mathcal{W}^{\prime}\subset \mathcal{U}$, $\mathcal{B}\in \mathfrak{C}\mid_{Y}$. Thus $\mathcal{B}\in \mathcal{C}$. Since $A=(A\cap Y)\cup B$, for all $A\in \mathcal{A}_{\mathcal{W}}$, clearly $\mathcal{A}_{\mathcal{W}}=\mathcal{B}\vee \{B\}\in \mathcal{C}$. So property iv) of Definition \ref{defasl} shows that $\mathcal{A}_{\mathcal{W}}\in \mathfrak{C}$. Thus $\mathcal{V}$ is a uniformly bounded cover of $X$. Since $\asdim_{\mathcal{C}}X\leq n$, there exist a uniformly bounded cover $\mathcal{V}^{\prime}$ of $X$ such that $\mathcal{V}$ refines $\mathcal{V}^{\prime}$ and the multiplicity of $\mathcal{V}^{\prime}$ is less than or equal to $n+1$. Let
$$\mathcal{U}^{\prime}=\{U\cap Y\mid U\in \mathcal{V}^{\prime}\,and\,V\subseteq U,\,for\,some\,V\in \mathcal{U}\}$$
Clearly $\mathcal{U}$ refines $\mathcal{U}^{\prime}$ and the multiplicity of $\mathcal{U}^{\prime}$ is less than or equal to $n+1$. In addition, it is straightforward to show that $\mathcal{U}^{\prime}$ is a uniformly bounded family of subsets of the LS.R space $(Y,\mathfrak{C}\mid_{Y})$. This shows that $\asdim_{\mathfrak{C}\mid_{Y}}Y\leq n$. Therefore, $\asdim_{\mathfrak{C}\mid_{Y}}Y\leq \operatorname{asdim}_{\mathfrak{C}}X$.
\end{proof}
\begin{theorem}
Assume that $(X,\mathfrak{C})$ and $(Y,\mathfrak{C}^{\prime})$ are two large scale equivalent LS.R spaces. Then
$$\asdim_{\mathfrak{C}}X=\asdim_{\mathfrak{C}^{\prime}}Y$$
\end{theorem}
\begin{proof}
Let $f:X\rightarrow Y$ be a large scale equivalence between $X$ and $Y$ and let $g:Y\rightarrow X$ be a large scale inverse of $f$. Assume that $n\in \mathbb{N}\cup \{0\}$ and $\asdim_{\mathfrak{C}}X\leq n$. Suppose that $\mathcal{U}$ is a uniformly bounded cover of $Y$. Let $\mathcal{V}=\{g(U)\mid U\in \mathcal{U}\}$. Suppose that $\mathcal{W}\subseteq \mathcal{V}$ and let $A\in \mathcal{A}_{\mathcal{W}}$. Let $\mathcal{W}^{\prime}=\{U\in \mathcal{U}\mid g(U)\in \mathcal{W}\}$. Assume that $A^{\prime}=g^{-1}(A)\cap (\bigcup_{U\in \mathcal{W}^{\prime}}U)$. It is easy to verify that $A^{\prime}\in \mathcal{A}_{\mathcal{W}^{\prime}}$ and $g(A^{\prime})=A$. Thus there exists a subset $\mathcal{B}$ of $\mathcal{A}_{\mathcal{W}^{\prime}}$ such that $g(\mathcal{B})=\mathcal{A}_{\mathcal{W}}$. Since $\mathcal{A}_{\mathcal{W}^{\prime}}\in \mathfrak{C}^{\prime}$, property ii) of Definition \ref{defasl} shows that $\mathcal{B}\in \mathfrak{C}^{\prime}$ and since $g$ is a LS.R mapping, $\mathcal{A}_{\mathcal{W}}=g(\mathcal{B})\in \mathfrak{C}$. Therefore $\mathcal{V}$ is a uniformly bounded cover of $Z$, where $Z=g(Y)$. By Proposition \ref{subs}, $\asdim_{\mathfrak{C}\mid_{Z}}Z\leq n$. Thus there exists a uniformly bounded cover $\mathcal{V}^{\prime}$ of $Z$ such that the multiplicity of $\mathcal{V}^{\prime}$ is less than or equal to $n+1$ and $\mathcal{V}$ refines $\mathcal{V}^{\prime}$. Let $\mathcal{U}^{\prime}=\{g^{-1}(V)\mid V\in \mathcal{V}^{\prime}\}$. It is straightforward to show that $\mathcal{U}$ refines $\mathcal{U}^{\prime}$ and the multiplicity of $\mathcal{U}^{\prime}$ is less than or equal to $n+1$. It remains to show that $\mathcal{U}^{\prime}$ is a uniformly bounded family of subsets of $(Y,\mathfrak{C}^{\prime})$. To do so, assume that $\mathcal{O}\subseteq \mathcal{U}^{\prime}$. Let $\mathcal{O}^{\prime}=\{V\in \mathcal{V}^{\prime}\mid g^{-1}(V)\in \mathcal{O}\}$. It is easy to show that $g(\mathcal{A}_{\mathcal{O}})\subseteq \mathcal{A}_{\mathcal{O}^{\prime}}$. So by property ii) of Definition \ref{defasl} we have $g(\mathcal{A}_{\mathcal{O}})\in \mathfrak{C}$ and since $f$ is a LS.R mapping, $f(g(\mathcal{A}_{\mathcal{O}}))\in \mathcal{C}^{\prime}$. Thus part i) of Lemma \ref{equi} implies that $\mathcal{A}_{\mathcal{O}}\in \mathfrak{C}^{\prime}$, as we desired. Therefore we showed that $\asdim_{\mathfrak{C}^{\prime}}Y\leq n$. This shows that $\asdim_{\mathfrak{C}^{\prime}}Y\leq \asdim_{\mathfrak{C}}X$. Similar arguments can verify the inequality $\asdim_{\mathfrak{C}}X\leq \asdim_{\mathfrak{C}^{\prime}}Y$.
\end{proof}
\begin{example}
Let $X=\mathbb{N}$ and consider $X$ with the discrete topology. Let $Y$ denote the one point compactification of $X$. Assume that $\mathcal{U}$ is a uniformly bounded cover of $(X,\mathfrak{C}_{Y})$. First notice that since $\mathcal{U}$ is uniformly bounded, Corollary \ref{ub} shows that each element of $\mathcal{U}$ is bounded and hence finite. Let $i\in \mathbb{N}$. Suppose that $\mathcal{W}=\{U\in \mathcal{U}\mid i\in U\}$. Since $\{i\}\in \mathcal{A}_{\mathcal{W}}$ and $\mathcal{A}_{\mathcal{W}}\in \mathfrak{C}_{Y}$, each element of $\mathcal{A}_{\mathcal{W}}$ is bounded. So all elements of $\mathcal{A}_{\mathcal{W}}$ are finite. Thus $\bigcup_{U\in \mathcal{W}}U\in \mathcal{A}_{W}$ is finite. It clearly implies that $\mathcal{W}$ is finite. Therefore we showed that the set of all $x\in \mathbb{N}$ such that $x,i\in U$, for some $U\in \mathcal{U}$, has maximum. Let $a_{1}=1$. Suppose that $a_{n}$ is chosen, for some $n\in \mathbb{N}$. Let $a_{n+1}$ be the largest $x\in \mathbb{N}$ such that $x,l\in U$, for some $U\in \mathcal{U}$ and some $l\leq a_{n}+1$. Assume that $V_{1}=\{1,...,a_{2}\}$ and $V_{n+1}=\{a_{n}+1,...,a_{n+2}\}$, for each $n\in \mathbb{N}$. Let $\mathcal{V}=\{V_{n}\mid n\in \mathbb{N}\}$. Assume that $U\in \mathcal{U}$ and $m=\min\{x\mid x\in U\}$. If $m=1$ then clearly $U\subseteq V_{1}$. Assume that $m\geq 2$. Choose $n\in \mathbb{N}$ such that $a_{n}+1\leq m\leq a_{n+1}$ (notice that it is immediate that $a_{n+1}\geq a_{n}+1$, for each $n\in \mathbb{N}$). It is easy to verify that $U\subseteq V_{n+1}$. Thus $\mathcal{U}$ refines $\mathcal{V}$. Suppose that $\mathcal{V}^{\prime}\subseteq \mathcal{V}$. If $\mathcal{V}^{\prime}$ is finite, then each element of $\mathcal{A}_{\mathcal{V}^{\prime}}$ is finite and if $\mathcal{V}^{\prime}$ is infinite, then each element of $\mathcal{A}_{\mathcal{V}^{\prime}}$ is infinite. Thus in both cases $\mathcal{A}_{\mathcal{V}^{\prime}}\in \mathfrak{C}_{Y}$. This shows that $\mathcal{V}$ is a uniformly bounded family of subsets of $X$. In addition, if $n\in \mathbb{N}$ then $V_{n}\cap V_{m}=\emptyset$, for each $m\geq n+2$. It shows that the multiplicity of $\mathcal{V}$ is equal to $2$. Thus we showed that $\asdim_{\mathfrak{C}_{Y}}X\leq 1$. Now let $\mathcal{O}=\{\{n,n+1\}\mid n\in \mathbb{N}\}$. It is easy to see that the family $\mathcal{O}$ is a uniformly bounded family of subsets of $X$. Assume that $\mathcal{O}^{\prime}$ is a uniformly bounded cover of $(X,\mathcal{C}_{Y})$ such that $\mathcal{O}$ refines it. Suppose that the multiplicity of $\mathcal{O}^{\prime}$ is less than or equal to $1$. Choose $O\in \mathcal{O}^{\prime}$ such that $\{1,2\}\subseteq O$. Suppose that $n\in \mathbb{N}$ and $n\in O$. Since the multiplicity of $\mathcal{O}^{\prime}$ is less than $1$, $n\notin O^{\prime}$, for all $O^{\prime}\in \mathcal{O}^{\prime}$ such that $O^{\prime}\neq O$. Since $\{n,n+1\}\in \mathcal{O}$, $n+1\in O$.  So $O=\mathbb{N}$ and this contradicts the fact $O$ is bounded. Thus the multiplicity of $\mathcal{O}^{\prime}$ is bigger than or equal to $2$. Therefore we showed that $\asdim_{\mathfrak{C}_{Y}}X\geq 1$ and this leads to $\asdim_{\mathfrak{C}_{Y}}X=1$.
\end{example}
\section{Large scale regular LS.R spaces}\label{s5}
\begin{definition}
We call the LS.R space $(X,\mathfrak{C})$ \emph{large scale regular} (\emph{LS-regular}), if $\mathcal{A}\in \mathfrak{C}$ and $(A_{1}\cup A_{2})\in \mathcal{A}$, for two nonempty subsets $A_{1}$ and $A_{2}$ of $X$, then there exist some $\mathcal{A}_{1},\mathcal{A}_{2}\in \mathfrak{C}$ such that they contain $A_{1}$ and $A_{2}$, respectively and $\mathcal{A}\subseteq (\mathcal{A}_{1}\vee \mathcal{A}_{2})$.
\end{definition}
\begin{corollary}
Let $X$ be a set and assume that $\mathcal{E}$ and $\lambda$ are denoting a coarse structure and an AS.R on $X$, respectively. Then $\mathfrak{C}_{\mathcal{E}}$, $\tilde{\mathfrak{C}}_{\mathcal{E}}$, $\mathfrak{C}_{\lambda}$ and $\tilde{\mathfrak{C}}_{\lambda}$ are LS-regular large scale resemblances on $X$.
\end{corollary}
\begin{proof}
It is straightforward.
\end{proof}
\begin{example}
Let $X$ and $\mathfrak{C}$ be as Example \ref{counter}. Clearly $\mathfrak{C}$ is not LS-regular.
\end{example}
\begin{example}
The LS.R $\mathfrak{C}_{Y}$ defined in Example \ref{count1}, is clearly LS-regular. Recall that $\mathfrak{C}_{Y}\neq \mathfrak{C}_{\mathcal{E}},\tilde{\mathfrak{C}}_{\mathcal{E}}$, for all coarse structure $\mathcal{E}$ on $X$.
\end{example}
\begin{definition}
Let $\mathfrak{C}$ be a LS.regular LS.R on the set $X$. For two subsets $A$ and $B$ of $X$, define\\
`$A\lambda_{\mathfrak{C}}B$ if $A,B\in \mathcal{A}$, for some $\mathcal{A}\in \mathfrak{C}$'.\\
It is easy to verify that $\lambda_{\mathfrak{C}}$ is an AS.R on the set $X$.
\end{definition}
Assume that $\lambda$ and $\lambda^{\prime}$ are two AS.R relations on the set $X$. We say $\lambda\leq \lambda^{\prime}$ if $A\lambda B$ implies $A\lambda^{\prime} B$, for all $A,B\subseteq X$. The following two propositions are easy to prove.
\begin{proposition}
Let $\lambda$ be an AS.R on the set $X$. Denote $\mathfrak{C}_{\lambda}$ and $\tilde{\mathfrak{C}}_{\lambda}$ by $\mathfrak{C}$ and $\mathfrak{C}^{\prime}$, respectively. Then $\lambda_{\mathfrak{C}}=\lambda$ and $\lambda_{\mathfrak{C}^{\prime}}\leq \lambda$.
\end{proposition}
\begin{proposition}
Let $\mathfrak{C}$ be a LS.regular LS.R on the set $X$ and denote $\lambda_{\mathfrak{C}}$ by $\lambda$. Then $\mathfrak{C}\subseteq \mathfrak{C}_{\lambda}$.
\end{proposition}
\begin{proposition}\label{111}
Let $(X,\mathfrak{C})$ be a LS-regular LS.R space. If $\mathcal{U}$ is a uniformly bounded family of subsets of the LS.R space $(X,\mathfrak{C})$, then it is a uniformly bounded family of subsets of the AS.R space $(X,\lambda_{\mathfrak{C}})$,
\end{proposition}
\begin{proof}
Assume that $A$ and $B$ are two subsets of $X$ such that $A\subseteq \bigotimes_{\mathcal{U}}(B)$ and $B\subseteq \bigotimes_{\mathcal{U}}(A)$. Let $\mathcal{V}=\{U\in \mathcal{U}\mid A\cap U\neq \emptyset\,and\,B\cap U\neq \emptyset\}$. Clearly $A,B\in \mathcal{A}_{\mathcal{U}}$ and hence $A\lambda_{\mathfrak{C}}B$.
\end{proof}
The inverse of Proposition \ref{111} is not true in general.
\begin{example}
Let $X=\mathbb{N}$ and assume that $\mathcal{E}$ denotes the family of all $E\subseteq X\times X$ with the following property.\\
`There exists some $n\in \mathbb{N}$ such that $E(x)$ and $E^{-1}(x)$ have at most $n$ members, for all $x\in X$.'\\
 The family $\mathcal{E}$ is a coarse structure on $X$ (Example 2.44 of \cite{Roe}). Let $\mathfrak{C}=\mathfrak{C}_{\mathcal{E}}$. For two subsets $A$ and $B$ of $X$ we have $A\lambda_{\mathfrak{C}}B$ if and only if $A$ and $B$ are both finite or both infinite (see Example 3.1 of \cite{me}). Let $U_{i}=\{i,i+1,...,2i\}$, for each $i\in \mathbb{N}$. Assume that $\mathcal{U}=\{U_{i}\mid i\in \mathbb{N}\}$. It is easy to verify that $\mathcal{U}$ is a uniformly bounded family of subsets of the AS.R space $(X,\lambda_{\mathfrak{C}})$. We claim that $\mathcal{A}_{\mathcal{U}}\notin \mathfrak{C}$. Suppose that, contrary to our claim, $E\in \mathcal{E}$ is such that $A\subseteq E(B)$, for all $A,B\in \mathcal{A}_{\mathcal{U}}$. Let $N\in \mathbb{N}$ be such that $E(x)$ and $E^{-1}(x)$ has at most $N$ members, for all $x\in X$. Clearly $X\in \mathcal{A}_{\mathcal{U}}$. Since for each $x\in X$, $E(x)$ and $E^{-1}(x)$ have finite elements, we can choose $m\in \mathbb{N}$ large enough to have $(E(m)\cup E^{-1}(m))\cap U_{i}=\emptyset$, for each $i\leq 2N$. Let $i>2N$. Since $U_{i}$ has more then $2N+1$ elements, we can choose $a_{i}\in U_{i}$ such that $a_{i}\notin (E(m)\cup E^{-1}(m))$. Let $A=(\cup_{i=1}^{2N}U_{i})\cup B$, where $B=\{a_{i}\mid i>2N\}$. Clearly $A\in \mathcal{A}_{\mathcal{U}}$ and $(E(m)\cup E^{-1}(m))\cap A=\emptyset$. It shows that $\mathbb{N}$ is not a subset of $E(A)$, a contradiction. Therefore $\mathcal{A}_{\mathcal{U}}\notin \mathfrak{C}$ and hence $\mathcal{U}$ is not a uniformly bounded family of subsets of the LS.R space $(X,\mathfrak{C})$.
\end{example}
\begin{definition}
We call the LS.R space $(X,\mathfrak{C})$ an A-LS.R space if,\\
i) $(X,\mathfrak{C})$ is LS-regular.\\
ii) $\mathcal{A}\subseteq \mathcal{P}(X)$ and $\{A,B\}\in \mathfrak{C}$, for each $A,B\in \mathcal{A}$, then $\mathcal{A}\in \mathfrak{C}$.
\end{definition}
We denote the category of all LS.R spaces and LS.R mappings by $\mathbf{L}$.
\begin{theorem}
Suppose that $\mathbf{A}$ and $\mathbf{R}$ denote full subcategories of $\mathbf{L}$ those objects are all A-LS.R spaces and all LS-regular LS.R spaces, respectively. Then,\\
i) $\mathbf{A}$ is isomorphic to the category of all AS.R spaces and AS.R mappings.\\
ii) $\mathbf{A}$ is a reflective full subcategory of $\mathbf{R}$.
\end{theorem}
\begin{proof}
i) Suppose that $F$ denotes the functor that associates each A-LS.R space $(X,\mathfrak{C})$, to the AS.R space $(X,\lambda_{\mathfrak{C}})$ and each LS.R mapping $f:X\rightarrow Y$ between two A-LS.R spaces $(X,\mathfrak{C})$ and $(Y,\mathfrak{C}^{\prime})$, to the AS.R mapping $f:(X,\lambda_{\mathfrak{C}})\rightarrow (Y,\lambda_{\mathfrak{C}^{\prime}})$. It is straightforward to show that $F$ defines an isomorphism of categories.\\
ii) Suppose that $(X,\mathfrak{C})$ is a LS-regular LS.R space. Define $\tilde{\mathfrak{C}}$ to be the set of all $\mathcal{A}\subseteq \mathcal{P}(X)$ such that $\{A,B\}\in \mathfrak{C}$, for all $A,B\in \mathcal{A}$. Clearly $(X,\tilde{\mathfrak{C}})$ is an A-LS.R space. In addition notice that if $(Y,\mathfrak{D})$ is an A-LS.R space and $f:(X,\mathfrak{C})\rightarrow (Y,\mathfrak{D})$ is a LS.R mapping, then $f:(X,\tilde{\mathfrak{C}})\rightarrow (Y,\mathfrak{D})$ is a LS.R mapping. Thus $((X,\tilde{\mathfrak{C}}),i)$ is the $\mathbf{A}$ reflection of $(X,\mathfrak{C})$.
\end{proof}

\end{document}